\documentclass[12pt]{extarticle}
\usepackage{amsmath, amsthm, amssymb, color}
\usepackage[colorlinks=true,linkcolor=blue,urlcolor=blue]{hyperref}
\usepackage{graphicx}
\usepackage{mathtools}
\usepackage{enumerate}
\usepackage{verbatim}
\usepackage{tikz,tikz-cd,tikz-3dplot}
\usetikzlibrary{matrix}
\usetikzlibrary{arrows}
\usepackage[ruled,linesnumbered]{algorithm2e}
\usepackage{caption}
\usepackage[normalem]{ulem}
\usepackage{subcaption}
\usepackage{multicol}
\oddsidemargin \evensidemargin
\usepackage{makecell}
\usepackage{array}
\usepackage{enumitem}
\usepackage{booktabs}

\newtheorem{theorem}{Theorem}[section]

\newtheorem{proposition}[theorem]{Proposition}
\newtheorem{lemma}[theorem]{Lemma}

\theoremstyle{definition}

\newtheorem{examplex}[theorem]{Example}
\newenvironment{example}
{\pushQED{\qed}\examplex}
{\popQED\endexamplex}

\newcommand{\PP}{\mathbb{P}}

\newcommand{\CC}{\mathbb{C} }

\newcommand{\NN}{\mathbb{N}}

\title{\bf Maximum Likelihood Estimation\\ on the Grassmannian of Lines}
\author{Hannah Friedman}
\date{}
\begin{document}
\maketitle
\begin{abstract}
  \noindent
  We study the positive Grassmannian through the lens of algebraic statistics.
A closed formula is presented for the maximum likelihood degree of the
Grassmannian of lines.
We study the real and positive critical points for the Grassmannian of lines in 3-space. 
\end{abstract}

 \let\thefootnote\relax

\section{Introduction}
The Grassmannian $\mathrm{Gr}(d,n)$ plays an important role as a statistical model in physics and probability.
In physics, it appears via the configuration space $X(d,n) \coloneq \mathrm{Gr}^\circ(d,n)/(\CC^*)^n$ of $n$ points in projective space $\PP^{d-1}$ \cite{ABFKST,BBB,ST}.
In statistics, it appears via the squared Grassmannian, the underlying model of projection determinantal point processes \cite{DFRS, friedman2025}.
The likelihood geometry of these two models was studied for $d = 2$ in \cite{BBB,ST} and \cite{friedman2025}, respectively. 

In this article, we study the positive Grassmannian, a central object in algebraic combinatorics; see, e.g., \cite{williams2023}.
We view the positive Grassmannian as a statistical model.
Consider a probability distribution on the $d$-subsets of $[n] = \{1,\ldots,n\}$, denoted $\binom{[n]}{d}$, where the probability of selecting a subset $I$ is proportional to the Pl\"ucker coordinate $x_I$.
In symbols,
\begin{align*}
  \mathrm{Prob}\big (Z = I\big ) \,\,=\,\, \frac{x_I}{\sum_{J \in \binom{[n]}{d}}x_J}, && \textrm{for all  }I \in \binom{[n]}{d}.
\end{align*}
The probability of an event must be a real nonnegative number. 
Therefore, a point $x = (x_I)_{I \in \binom{[n]}{d}}\in \mathrm{Gr}(d,n)_{\mathbb R}\subseteq \mathbb P^{\binom{n}{d}-1}_{\mathbb R}$ defines a probability distribution when all coordinates have the same sign.
When this is the case, we say $x$ is a point in the \emph{positive Grassmannian}
 \[\mathrm{Gr}(d,n)_{>0} = \{x \in \mathrm{Gr}(d,n)_{\mathbb R} : \mathrm{sgn}(x_{I}) > 0 \textrm{ for all } I \subseteq [n],\, |I| = d\}.\]

In likelihood geometry, a data point $u \in \NN^{\binom{n}{d}}$ represents a sample from an unknown probability distribution in our model $\mathrm{Gr}(d,n)_{>0}$.
We estimate this unknown distribution by computing the maximizer of the \emph{log-likelihood function}
\begin{equation*}%\label{eq:Lu}
  L_u(x) = \sum_{I \in \binom{[n]}{d}} u_I\log(x_I) - (\sum_{I \in \binom{[n]}{d}} u_I)\log(\sum_{I \in \binom{[n]}{d}} x_I)
\end{equation*}
subject to the constraint that $x$ is in the positive Grassmannian $\mathrm{Gr}(d,n)_{>0}$.
The maximizer $x^*$ is called the \emph{maximum likelihood estimate}, and it is the distribution that was most likely sampled from if one observed the data point $u$. 
To solve this optimization problem accurately, one must compute all critical points of the problem.
The number of complex critical points is finite and constant for generic data $u$; this number is called the \emph{maximum likelihood degree} (ML degree) \cite{CHKS,HS}. 
The ML degrees of the Grassmannian were first computed for $d=2$ and $n=4,5$ in \cite[p. 400]{HKS}.
More ML degrees were found numerically in \cite[Theorem 4.1]{DFRS}.
Our main result is a closed formula for the ML degree of the Grassmannian of lines ($d=2$):
\begin{theorem}\label{thm:grass-ml-deg}
  The Grassmannian $\mathrm{Gr}(2,n)$ has ML degree $(2^{n-1} - n)(n-3)!$.
\end{theorem}
Theorem~\ref{thm:grass-ml-deg} is a partial solution to \cite[Problem 13]{HKS}.
The ML degree is an algebraic measure of the complexity of the maximum likelihood estimation problem.
Its factorial growth in $n$ indicates that maximum likelihood estimation is difficult for the Grassmannian.

The ML degree is now known for three statistical models on $\mathrm{Gr}(2,n)$, namely the positive Grassmannian $\mathrm{Gr}(2,n)_{>0}$ (Theorem~\ref{thm:grass-ml-deg}), the moduli space $X(2,n) = \mathcal M_{0,n}$ (\cite[Proposition 1]{ST}), and the squared Grassmannian~$\mathrm{sGr}(2,n)$ (\cite[Theorem 1.2]{friedman2025}):
\[
\begin{array}{cccccc}
  \toprule
  & \mathrm{Gr}(2,n)_{>0} && \mathcal M_{0,n} && \mathrm{sGr}(2,n)\\
  \midrule
  \textrm{ML degree} & (n-3)!(2^{n-1} - n) && (n-3)! && (n-1)!/2\\
  \bottomrule
\end{array}
\]

In Section~\ref{sec:proof}, we prove Theorem~\ref{thm:grass-ml-deg}.
We conclude the section with comments on the case $d = 3, n = 6$. 
In particular, we use numerical methods to compute two components of the intersection of the Grassmannian $\mathrm{Gr}(3,6)$ with the $A$-discriminant of the $(3,6)$ hypersimplex (Theorem~\ref{thm:3,6}).
%In Section~\ref{sec:positive}, we consider the reality and positivity of the critical point set when the data $u$ lies in the probability simplex $\Delta_{\binom n 2 - 1}$ with a focus on the case $n = 4$.
% We compute the logarithmic discriminant and discuss some of the regions in its complement. 
In Section~\ref{sec:positive}, we use numerical algebraic geometry to investigate the number of real and positive critical points of the maximum likelihood estimation problem on the Grassmannian with a focus on $\mathrm{Gr}(2,4)$.
We compute the logarithmic discriminant of $\mathrm{Gr}(2,4)$ and exhibit examples of data points for which the log-likelihood function has multiple positive critical points (Figure~\ref{fig:3crit}).
The software accompanying this paper is available at
\begin{center}
  \url{https://zenodo.org/records/21470087}.
\end{center}

\section{The ML Degree of the Grassmannian of Lines}\label{sec:proof}
We will determine the ML degree of the Grassmannian of lines by computing an Euler characteristic.
We explain the setup for general $d$ and $n$.
By \cite[Theorem 1.7]{HS}, the ML degree of $\mathrm{Gr}(d,n)$ is equal to the signed Euler characteristic of the very affine variety $\mathrm{Gr}(d,n)\backslash \mathcal H$ where $\mathcal H = \{x \in \PP^{\binom{n}{d}-1}: \prod_I x_I \cdot \sum_I x_I = 0\}$.
In symbols, \[\mathrm{MLdegree}(\mathrm{Gr}(d,n)) = (-1)^{d(n-d)}\chi(\mathrm{Gr}(d,n)\backslash \mathcal H).\]
In order to compute the Euler characteristic of this very affine variety, we will use the fact that the ML degree of the configuration space $X(d,n)$ is known for $d = 2$.
Let $\mathrm{Gr}(d,n)^\circ = \mathrm{Gr}(d,n)\backslash \{x \in \PP^{\binom{n}{d}-1}: \prod_I x_I =0\}$ denote the open Grassmannian. 
Taking the quotient of the open Grassmannian by the algebraic torus, we obtain the configuration space $X(d,n) = \mathrm{Gr}(d,n)^\circ/(\mathbb C^*)^n$. 
By the definition of $X(d,n)$, there exists a quotient map 
\begin{equation}\label{eq:quotient}
  \mathrm{Gr}(d,n) \backslash \mathcal H = \mathrm{Gr}(d,n)^\circ \backslash V(\sum_{I \in \binom{[n]}{d}}x_{I}) \longrightarrow X(d,n) = \mathrm{Gr}(d,n)^\circ/(\CC^*)^n,\qquad
  x \mapsto [x].
\end{equation}
Writing $f_x(\lambda) = \sum_{I \in \binom{[n]}{d}} x_I \prod_{i \in I}\lambda_i$ for the polynomial in $\lambda = (\lambda_i)_{i\in [n]}$ with coefficients $x$, the fiber of the point $[x] \in X(d,n)$ under \eqref{eq:quotient} is the set
\begin{align*}
  \mathcal F_x &= \{\lambda \cdot x : \lambda \in (\mathbb C^*)^n/\mathbb C^*,\,\,f_x(\lambda)
  \neq 0\} \subseteq \mathrm{Gr}(d,n)^\circ.
\end{align*}
Here $\lambda \cdot x$ denotes the point $(x_I \cdot \prod_{i \in I}\lambda_i)_{I \in \binom{[n]}{d}}$.
Note that the map \eqref{eq:quotient} is surjective, since $f_x(\lambda) = 0$ for all $\lambda$ implies that $x = 0$. 
If all fibers $\mathcal F_x$ have the same Euler characteristic, then, by multiplicativity of the Euler characteristic, 
\begin{equation}\label{eq:factor}
  \chi(\mathrm{Gr}(d,n) \backslash \mathcal H) =
  \chi(X(d,n)) \cdot \chi(\mathcal F_x). 
\end{equation}
If the Euler characteristic is not constant on fibers, then correction terms are added.
The fiber $\mathcal F_x$ is isomorphic to the set
\begin{equation}\label{eq:v-aff-hyper}
  \{\lambda \in (\mathbb C^*)^n/\mathbb C^* : f_x(\lambda)  \neq 0\}. 
\end{equation}
The Newton polytope of $f_x(\lambda)$ is the $(d,n)$-hypersimplex, denoted $\Delta_{d,n} = \mathrm{conv}(\sum_{i\in I}e_i : I \in \binom{[n]}{d})$ where $e_i$ is the $i$th standard basis vector.
For generic $x \in \mathbb P^{\binom n d -1}$, the Euler characteristic of $\mathcal F_x$ is the signed normalized volume of $\Delta_{d,n}$. 
By \cite[Theorem 2]{ABB}, the locus of coefficient vectors $x \in \mathbb P^{\binom n d - 1}$ where the Euler characteristic of $\mathcal F_x$ drops is the zero set of the principal $A$-determinant \cite[Chapter 10]{GKZ} of $\Delta_{d,n}$.  
Recall that the coefficients of $f_x$ are the Pl\"ucker coordinates of points in the open Grassmannian.
Therefore \eqref{eq:factor} holds if the principal $A$-determinant of $\Delta_{d,n}$ does not intersect the open Grassmannian $\mathrm{Gr}(d,n)^\circ$. 

We now turn to the case $d = 2$.
The configuration space $X(2,n) = \mathcal M_{0,n}$ is the moduli space of $n$ marked points in $\mathbb P^1$. 
We will prove that the intersection of the principal $A$-determinant of $\Delta_{2,n}$ with the open Grassmannian $\mathrm{Gr}(2,n)^\circ$ is empty. 
By \cite[Theorem 3.6]{CHKO}, the principal $A$-determinant of $\Delta_{2,n}$ is the product of principal minors of size at least $4$ of 
  \begin{equation}\label{eq:symmetric}
    X = \begin{pmatrix}
      0 & x_{12}& x_{13}& \cdots & x_{1n}\\
      x_{12} & 0 & x_{23} & \cdots & x_{2n}\\
      x_{13} & x_{23} & 0 & \cdots & x_{3n}\\
      \vdots & \vdots & \vdots & \ddots & \vdots\\
      x_{1n} & x_{2n} & x_{3n} & \cdots & 0
    \end{pmatrix}.
  \end{equation}
The bulk of this section is dedicated to proving that the principal minors of this matrix may be written as monomials in $x$ if $x$ is a point on the Grassmannian. 
\begin{proposition}\label{prop:A-det}
Suppose $x \in \mathrm{Gr}(2,n)$, and that $X$ is the symmetric matrix \eqref{eq:symmetric}.
Then $\det(X) = (-1)^{n-1}2^{n-2} x_{12}x_{23} \cdots x_{n-1, n}x_{1n}$.
In particular, the principal $A$-determinant of the second hypersimplex $\Delta_{2,n}$ does not intersect the open Grassmannian $\mathrm{Gr}(2,n)^\circ$. 
\end{proposition}
\begin{proof}
  Let $X'$ denote the skew symmetric matrix
  \begin{equation*}
    \begin{pmatrix}
      0 & x_{12}& x_{13}& \cdots & x_{1n}\\
      -x_{12} & 0 & x_{23} & \cdots & x_{2n}\\
      -x_{13} & -x_{23} & 0 & \cdots & x_{3n}\\
      \vdots & \vdots & \vdots & \ddots & \vdots\\
      -x_{1n} & -x_{2n} & -x_{3n} & \cdots & 0
    \end{pmatrix}.
  \end{equation*}
  Since $x \in \mathrm{Gr}(2,n)$, the matrix $X'$ has rank $2$, and thus there exist $u,v \in \CC^n$ such that we can write the $(i,j)$-entry of $X'$ as $u_iv_j - u_jv_i$.
  Then the entries of $X$ are $X_{ij} =- (-1)^{i<j} (u_iv_j - u_jv_i)$ where $(-1)^{i<j}$ is $-1$ if $i < j$ and $1$ otherwise. 
  Therefore the determinant of $X$ is
  \begin{equation*}
    \det(X) = (-1)^n\sum_{\sigma \in \mathfrak S_n} \mathrm{sgn}(\sigma) \prod_{i=1}^n (-1)^{i<\sigma(i)} (u_iv_{\sigma(i)} - u_{\sigma(i)}v_i).
  \end{equation*}
  where $ \mathfrak S_n$ denotes the symmetric group on $n$ elements and $\mathrm{sgn}(\sigma)$ denotes the sign of a permutation $\sigma$.
  An index $i$ with $i < \sigma(i)$ is called an \emph{excedance} of $\sigma$, and we write $\mathrm{exc}(\sigma) = |\{i \in [n]: i < \sigma(i)\}|$ for the number of excedances in the permutation $\sigma$.
  For a subset $S \subseteq [n]$, we write $u_S = \prod_{i \in S}u_i$.
  Expanding the determinant gives
  \begin{align*}
    \det(X) &= (-1)^n \sum_{S \subseteq [n]}(-1)^{|S^c|} \sum_{\sigma \in \mathfrak S_n}\mathrm{sgn}(\sigma)(-1)^{ \mathrm{exc}(\sigma)} u_Sv_{\sigma(S)}u_{\sigma(S)^c}v_{S^c}\\
    &= (-1)^n \sum_{S \subseteq [n]} \sum_{\substack{T \subseteq [n] \\|T| = |S|}} (-1)^{|S^c|} \left (\sum_{\substack{\sigma \in \mathfrak S_n\\ \sigma(S) = T}}\mathrm{sgn}(\sigma)(-1)^{\mathrm{exc}(\sigma)}  \right )u_Sv_{T}u_{T^c}v_{S^c}
  \end{align*}
  where $S^c$ and $T^c$ denote the set complements $[n] \backslash S$ and $[n] \backslash T$.

  By Lemma~\ref{lem:coeffs}, if $S$ and $T$ are subsets of $[n]$, then the coefficients evaluate to
  \begin{equation*}
    \sum_{\substack{\sigma \in \mathfrak S_n\\ \sigma(S) = T}} \mathrm{sgn}(\sigma)(-1)^{\mathrm{exc}(\sigma)} =
    \begin{cases}
      2^{n-1} &\textrm{if } T = S = [n] \textrm{  or  } T = S = \emptyset\\
      2^{n-2} &\textrm{if } \emptyset \subsetneq T = S \subsetneq [n] \textrm{  or  } T = S+1;\\
      0 &\textrm{else.}
    \end{cases}
  \end{equation*}
  Here $S + 1$ denotes the set $\{(s \bmod n) + 1: s \in S\}$. 
  In our determinant, the monomial $u_{[n]}v_{[n]}$ appears whenever $S = T$.
  Thus its coefficient is
  \begin{equation}\label{eq:coeff}
    \begin{cases}
      2^n + 2^{n-2}\sum_{\emptyset \subsetneq S \subsetneq [n]} (-1)^{|S^c|} = 2^{n-1} &\textrm{if $n$ is even and}\\
      2^{n-2}\sum_{\emptyset \subsetneq S \subsetneq [n]} (-1)^{|S^c|} = 0 &\textrm{if $n$ is odd}.
    \end{cases}
  \end{equation}
  In either case, the determinant evaluates to 
  \begin{align*}
    \det(X) &= 
    (-1)^n \sum_{\emptyset \subseteq S \subseteq [n]} (-1)^{|S^c|}2^{n-2}u_Sv_{S+1}u_{{S^c+1}}v_{S^c}\\
    &= (-2)^{n-2}\prod_{i = 1}^{n} (u_iv_{(i\bmod n)+1} - u_{(i\bmod n)+1}v_{i})
    =
    -(-2)^{n-2} x_{12} x_{23} \cdots x_{n-1,n}x_{1n}.
  \end{align*}
  The final sign flip comes from $x_{1n} = -x_{n1}$. 
  It follows that if $\det(X)$ is zero, then some $x_{ij}$ must be zero and thus the vector $x$ cannot lie in the open Grassmannian.
  For the final claim, we note that every principal submatrix of $X$ having size at least $4 \times 4$ is a point in a smaller Grassmannian. 
  Since the principal $A$-determinant of $\Delta_{2,n}$ is the product of principal minors of $X$ having size at least 4 \cite[Theorem 3.6]{CHKO}, we may write it as a monomial on the Grassmannian, and hence it does not intersect $\mathrm{Gr}(2,n)^\circ$. 
\end{proof}

In the proof of Proposition~\ref{prop:A-det}, we computed the coefficients \eqref{eq:coeff} by evaluating a restricted version of the signed excedance enumerator.
By \cite[p. 785]{siva2011}, the \emph{signed excedance enumerator} $\sum_{\sigma \in \mathfrak S_n}(-1)^{\mathrm{inv}(\sigma)}q^{\mathrm{exc}(\sigma)} = (1-q)^{n-1}$ is the determinant of the matrix
\begin{equation}\label{eq:exc}
      Q = 
  \begin{pmatrix}
    1 & q &q &  \cdots & q\\
    1 & 1 & q& \cdots & q\\
    \vdots & \vdots &  \vdots & & \vdots\\
    1 & 1 & 1& \cdots & 1
  \end{pmatrix}.
  \end{equation}
We will give a determinantal formula for the restricted excedance enumerator $\sum_{\substack{\sigma \in \mathfrak S_n\\ \sigma(S) = T}} \mathrm{sgn}(\sigma) q^{\mathrm{exc}(\sigma)}$ defined by $S$ and $T$, where $S$ and $T$ are subsets of $[n]$ having the same size.
The sum is now only over permutations that map the elements of $S$ to the elements of $T$. 
Let $Q_{S,T}$ denote the matrix $Q$ with the entries not in $S \times T$ or $S^c \times T^c$ zeroed out, where $S^c$ and $T^c$ denote the set complements of $S$ and $T$, respectively.
Let $S + 1$ denote the set $\{(s\bmod n)+1:s \in S\}$.
\begin{lemma}\label{lem:coeffs}
  Given subsets $\emptyset \subseteq S,T \subseteq [n]$,  with $|S| = |T|$, the restricted signed excedance enumerator is equal to
  \begin{equation*}
    \sum_{\substack{\sigma \in \mathfrak S_n\\ \sigma(S) = T}} \mathrm{sgn}(\sigma) q^{\mathrm{exc}(\sigma)} =
    \det(Q_{S,T}) =
    \begin{cases}
      (1-q)^{n-1} &\textrm{if } T = S = [n] \textrm{ or } T = S = \emptyset\\
      (1-q)^{n-2} &\textrm{if } T = S,\\
      -q(1-q)^{n-2} &\textrm{if } T = S+1,\\
      0 & \textrm{otherwise.}
    \end{cases}
  \end{equation*}
\end{lemma}
  \begin{example}\label{ex:QST}
  Let $n = 4$ and $S = \{1,2\}$. The matrices $Q_{S,T}$ and their determinants are shown below for $T = S$ and $T = S + 1$, and $T \neq S, S + 1$:
  \begin{gather*}
    \begin{matrix}
      T && \{1,2\} &&  \{2,3\} && \{2,4\}\\
      \\
      Q_{S,T} && 
      \begin{pmatrix}
      1 & q & 0 & 0\\
      1 & 1 & 0 & 0\\
      0 & 0 & 1 & q\\
      0 & 0 & 1 & 1
    \end{pmatrix}
    &&
    \begin{pmatrix}
      0 & q & q & 0\\
      0 & 1 & q & 0\\
      1 & 0 & 0 & q\\
      1 & 0 & 0 & 1
    \end{pmatrix}
    &&
    \begin{pmatrix}
      0 & q & 0 & q\\
      0 & 1 & 0 & q\\
      1 & 0 & 1 & 0\\
      1 & 0 & 1 & 0
    \end{pmatrix}\\ \\
    \det(Q_{S,T}) &&
    (1 - q)^2 && -q(1 - q)^2 && 0
    \end{matrix}\qedhere
  \end{gather*}
\end{example}

  \begin{proof}
    The case $T = S = [n]$ or $T = S = \emptyset$ recovers the classical signed excedance enumerator, and the determinantal form is given in \cite{siva2011}. 
  Using the Leibniz expansion of the determinant, we see that taking $\det(Q_{S,T})$ restricts the sum to those permutations mapping $S$ to $T$.
  It suffices to prove that $\det(Q_{S,S}) = (1-q)^{n-2}$,  $\det(Q_{S, S+1}) = -q(1-q)^{n-2}$, and $\det(Q_{S,T}) = 0$ for all other $T$.
  In the $\emptyset \subsetneq T = S \subsetneq [n]$ case, $Q_{S,T}$ is block diagonal with two blocks of the form \eqref{eq:exc} having sizes $|S|$ and $n - |S|$.
  Thus the determinant is $(1-q)^{|S|-1}(1 - q)^{n - |S| - 1} = (1-q)^{n-2}$; see \cite[Theorem 1]{siva2011}.

  The matrix $Q_{S,S+1}$ has two different types of blocks.
  One has size $|S|$ and occupies the entries $S \times T$; the other has size $n - |S|$ and occupies the entries $S^c \times T^c$. 
  Depending on whether $n$ is in $S$ or $S^c$, one block has the form \eqref{eq:exc}, and the other has the form
    \begin{equation*}
    \begin{pmatrix}
      q & q & q & \cdots & q\\
      1 & q & q & \cdots & q\\
      \vdots & \vdots & \vdots & & \vdots\\
      1 & 1& 1 & \cdots & q
    \end{pmatrix}.
  \end{equation*}
  In either case, the determinant of $Q_{S,T}$ is $-q(1-q)^{n-2}$; see \cite[Theorem 6]{siva2011}. 

  We now argue that $Q_{S,T}$ is singular if $T \notin\{ S, S+1\}$.
  If the $S \times T$ block is singular, then $Q_{S,T}$ is singular.
  We therefore assume that the $S \times T$ block is nonsingular and argue that the $S^c \times T^c$ block is singular. 
  We order the elements of $S$ as $s_1 < \ldots < s_r$.
  Note that the $|S| \times n$ submatrix formed by taking the rows of $Q$ indexed by $S$ has the property that the columns $s_i+1, s_i+2, \ldots, s_{i+1}$ are pairwise dependent for $i \in [r]$. 
  Since the $S \times T$ block is nonsingular, every cyclic interval $[s_i+1, s_{i+1}]$ contains precisely one element $t_i$ of $T$.
  If $t_i = s_i+1$ for all $i$, then $T = S+1$, and if $t_i = s_{i+1}$, then $T = S$.
  If $T \notin\{ S, S+1\}$, then there exists some $i$ such that $s_{i}, s_{i}+1\in T^c$.
  The columns of $Q$ indexed by $s_i$ and $s_i+1$ differ only in the $s_i$th row.
  Since $s_i \notin S^c$, the $s_i$ and $s_i+1$ columns in $Q_{S,T}$ are dependent.
  Thus $\det(Q_{S,T}) = 0$.
\end{proof}
\begin{proof}[Proof of Theorem~\ref{thm:grass-ml-deg}]
  Since $\mathrm{Gr}(2,n)$ is smooth, we apply \cite[Theorem 1.7]{HS} and show that $\chi(\mathrm{Gr}(2,n) \backslash \mathcal H) = (2^{n-1} - n)(n-3)!$.
  The base $\mathcal M_{0,n}$ of the map
  \begin{align*}
    \mathrm{Gr}(2,n)^\circ \backslash V(\sum_{1 \leq i < j \leq n}x_{ij}) \longrightarrow \mathcal M_{0,n},&&    x \mapsto [x]
  \end{align*}
  has Euler characteristic $\chi(\mathcal M_{0,n}) = (-1)^{n-3}(n-3)!$ by \cite[Proposition~1]{ST}.
  Since the Euler characteristic is multiplicative for maps of complex algebraic varieties (see \cite[Theorem 3.2.3]{ACM}), it suffices to show that all fibers  have Euler characteristic $(-1)^{n-1}(2^{n-1}-n)$.
  By Proposition~\ref{prop:A-det} and \cite[Theorem 6.2.4]{GKZ}, the fiber $\mathcal F_x$ has Euler characteristic $\chi(\mathcal F_x) = (-1)^{n-1}\mathrm{Vol}(\Delta_{2,n}) = (-1)^{n-1}(2^{n-1} - n)$ for all $x \in \mathcal M_{0,n}$.
  Thus $\chi(\mathrm{Gr}(2,n)^\circ \backslash V(\sum_{1 \leq i < j \leq n}x_{ij})) = \chi(\mathcal M_{0,n}) \cdot \chi(\mathcal F_x) = (n-3)!(2^{n-1} - n)$, as desired.
\end{proof}

The proof does not extend directly to the Grassmannian $\mathrm{Gr}(d,n)$.
For $d > 2$, the principal $A$-determinant of $\Delta_{d,n}$ does intersect the open Grassmannian.
We demonstrate this in the example $d = 3, n = 6$.

  The configuration space $X(3,6)$ has ML degree and Euler characteristic $26$ \cite[Example 5.2]{ABFKST}.
  The generic fiber has Euler characteristic $-\mathrm{Vol}(\Delta_{3,6}) = -66$.
  The product is $-1\,716$, whereas the ML degree of $\mathrm{Gr}(3,6)$ is $1\,937$ \cite[Theorem 4.1]{DFRS}.
  The discrepancy comes from the fact that the principal $A$-determinant of the $(3,6)$-hypersimplex intersects the open Grassmannian $\mathrm{Gr}(3,6)^\circ$.
  By \cite[Theorem 10.1.2]{GKZ}, this principal $A$-determinant is the product of $A$-discriminants of simplices, second hypersimplices, and of the $(3,6)$-hypersimplex.
  The $A$-discriminant of a simplex is either $1$ or the variable $x_{ijk}$, so these do not intersect the open Grassmannian.
  Since $\mathrm{Gr}(3,6)$ contains smaller Grassmannians, by Proposition~\ref{prop:A-det}, the $A$-discriminants of the second hypersimplices do not intersect the open Grassmannian $\mathrm{Gr}(3,6)^\circ$.
  Therefore, the only possible intersection is with the $A$-discriminant of $\Delta_{3,6}$.
  This $A$-discriminant is a hypersurface in $\mathbb P^{19}$ of degree $96$ \cite[p. 236]{HS18}. 

Using numerical methods, we compute the intersection of the Grassmannian $\mathrm{Gr}(3,6)$ with the $A$-discriminant of $\Delta_{3,6}$.
Write $I_{\mathrm{Gr}(3,6)} \subseteq R := \mathbb R[x_{ijk}: 1 \leq i < j < k \leq n]$ for the prime ideal of $\mathrm{Gr}(3,6)$, which is minimally generated by $35$ Pl\"ucker quadrics. 
\begin{theorem}\label{thm:3,6}
  The intersection in  $\mathbb P^{19}$ of the $A$-discriminant of $\Delta_{3,6}$ with the Grassmannian $\mathrm{Gr}(3,6)$ contains the two irreducible components
  \begin{align}\label{eq:components}
    V(I_{\mathrm{Gr}(3,6)} + \langle x_{123}x_{456} - x_{124}x_{356} \rangle)
    && \textrm{and}
    &&
    V(I_{\mathrm{Gr}(3,6)} + \langle x_{126}x_{345} - x_{136}x_{245} \rangle),
  \end{align}
  each of which has dimension $8$ and degree $84$.
\end{theorem}
\begin{proof}
  The $A$-discriminant of $\Delta_{3,6}$ is the set of coefficients $x$ such that the derivatives of $f_x(\lambda) = \sum_{1 \leq i < j < k \leq n}x_{ijk} \lambda_i\lambda_j \lambda_k$ have a common root.
  We aim to find the vectors $(x_{ijk})$ in the Grassmannian with this property.
  Since the $A$-discriminant is invariant under the action of the torus $(\mathbb C^*)^n$, it makes sense to say that an equivalence class $[x] \in X(3,6)$ is in the $A$-discriminant.
  We therefore suppose our coefficients $x$ are in the model for $X(3,6)$ parametrized by matrices of the form
  \begin{equation}\label{eq:X36}
    \begin{pmatrix}
      1 & 0 & 0 & 1 & 1 & 1\\
      0 & 1 & 0 & 1 & t_1 & t_2\\
      0 & 0 & 1 & 1 & t_3 & t_4
    \end{pmatrix}. 
  \end{equation}
  We write $x(t)$ for the vector of $3 \times 3$ minors of \eqref{eq:X36}.
  A vector $x(t)$ is in the $A$-discriminant if there exists $\lambda \in (\mathbb C^*)^6$ such that the pair $(t, \lambda)$ is in the incidence variety
  \begin{align}\label{eq:incidence}
    \overline{\{(t, \lambda) \in (\mathbb C^*)^4 \times (\mathbb C^*)^6 : \frac{\partial f_{x(t)}}{\partial \lambda_i}(\lambda) = 0 \textrm{ for } i = 1, \ldots 6\}}.
  \end{align}

  The proof proceeds by numerical and symbolic computations which are available on the \verb+Zenodo+ page accompanying this paper. 
  A numerical irreducible decomposition performed using \verb+HomotopyContinuation.jl+ \cite{hc} reveals that \eqref{eq:incidence} has two irreducible components of dimension 4.
  Therefore there exists a polynomial $p \in R$ such that $\mathrm{Gr}(3,6) \cap V(p)$ is contained in our intersection and has two reduced components. 
  It is possible that the intersection of $\mathrm{Gr}(3,6)$ with the $A$-discriminant has additional nonreduced components that are not detected by the numerical irreducible decomposition.

  To compute the degree of $p$, we intersect \eqref{eq:incidence} with a product $L_1 \times L_2$ of general linear spaces where $\dim(L_1) = 1$ and $\dim(L_2) = 5$ using \cite{hc}.
  Such a product of linear spaces has $4$ intersection points with \eqref{eq:incidence}. 
  Since the $A$-discriminant of $\Delta_{3,6}$ is invariant under the torus action, it is homogeneous with respect to the $\mathbb Z^6$ grading $\deg(x_{ijk}) = e_i + e_j + e_k$ where $e_i$ is the $i$th standard basis vector.
  In particular, the multidegree of the $A$-discriminant is a multiple of the all ones vector.
  Since $p$ has total degree $4$, the multidegree of $p$ is $(2,2,2,2,2,2)$. 
  The vector space of degree $(2,2,2,2,2,2)$ polynomials in $R$ has dimension $85$.
  By intersecting with general products of linear spaces $L_1 \times L_2$ as above, we sample $100$ points on the discriminant and create a $100 \times 85$ matrix where each entry is a degree $(2,2,2,2,2,2)$ monomial evaluated at one of these points.
  The kernel of this matrix has dimension $70$, with $69$ of these dimensions coming from the Pl\"ucker relations.
  The remaining kernel direction gives the coefficients of $p$, and a primary decomposition of $I_{\mathrm{Gr}(3,6)} + \langle p \rangle$ in \cite{oscar} reveals the binomial factors.
  The degree follows from a symbolic computation.
\end{proof}

As mentioned in the proof, current implementations of numerical irreducible decomposition can only detect reduced components of a scheme.
The ideal defining \eqref{eq:incidence} may have additional primary components that are not detected by the numerical irreducible decomposition.
Thus, the intersection of $\mathrm{Gr}(3,6)$ with the $A$-discriminant of $\Delta_{3,6}$ may have more components than \eqref{eq:components}.

Numerical computations suggest that for a generic $x$ in one of the two components \eqref{eq:components}, the fiber $\mathcal F_x$ of $x$ has Euler characteristic $\chi(\mathcal F_x) = -62$, and that for a generic $x$ in the intersection of the two components $\chi(\mathcal F_x) = -58$.
Each of these two components has Euler characteristic $-12$ and the intersection has Euler characteristic $5$.
If this were the full stratification, then using \cite[Lemma 2.3]{ABFKST}, the Euler characteristic would be $-1\,812$ instead of $-1\,937$.
More work is needed to determine the full stratification such that the Euler characteristic is constant on the fibers and whether there are additional components in the intersection.

\section{Real and Positive Critical Points}\label{sec:positive}
We now consider the scenario where the data point $u$ is in the probability simplex $\Delta_{\binom n d - 1}$.
In particular, we assume that all coordinates of $u$ are positive.
For the squared Grassmannian $\mathrm{sGr}(2,n)$ and the configuration space $\mathcal M_{0,n}$, this assumption implies that all critical points are real; see \cite[Proposition 1]{ST} and \cite[Theorem 1.3]{friedman2025}.
For the squared Grassmannian, all critical points are also positive, meaning that they all lie in the probability simplex.
On the other hand, viewing $\mathcal M_{0,n}$ as a linear model as in \cite{ST}, the log-likelihood function has at most one positive critical point by \cite[Proposition 7.3.4]{sullivant}. 

\begin{figure}[h! tbp]
  \centering
  \begin{subfigure}[b]{0.48\textwidth}
    \centering
    \caption*{$n = 6$}
    \includegraphics[width=0.8\textwidth]{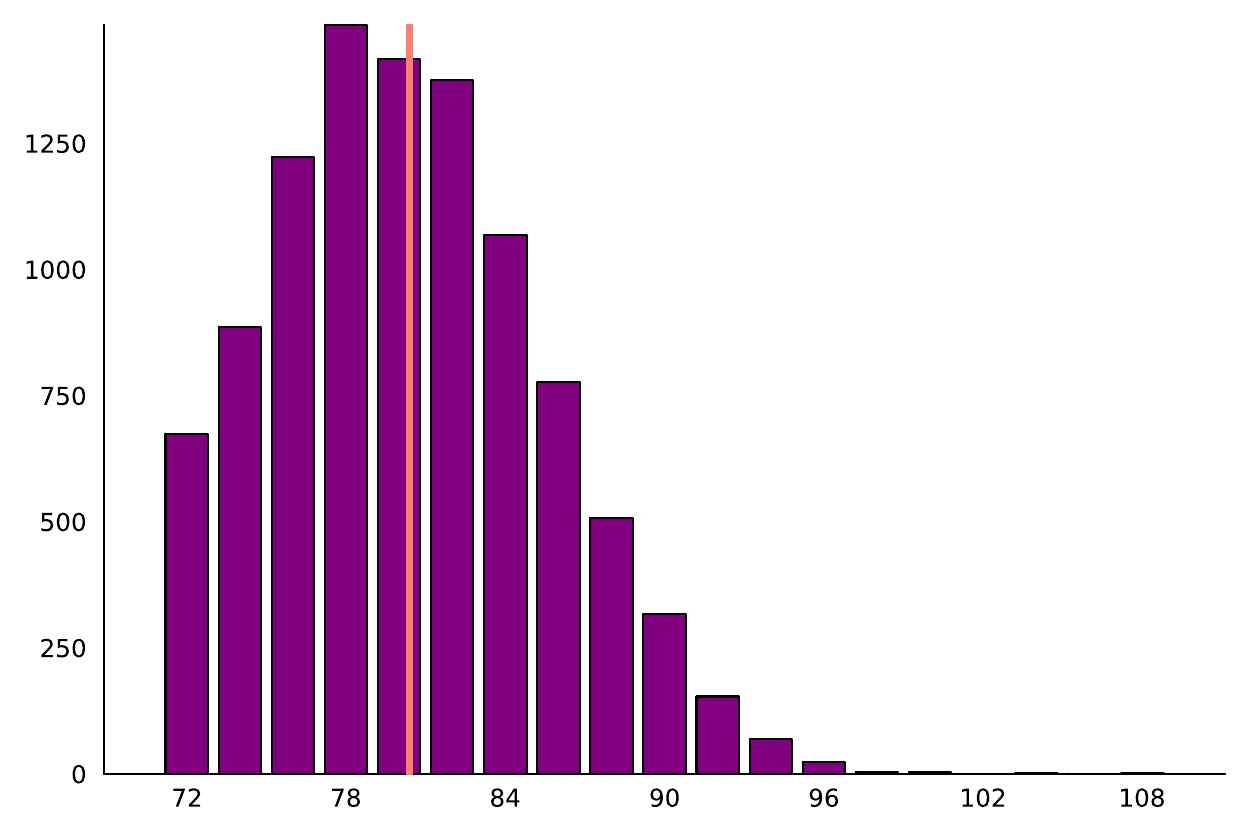}
    \label{fig:panel_a}
  \end{subfigure}
  \begin{subfigure}[b]{0.48\textwidth}
    \centering
    \caption*{$n = 7$}
    \includegraphics[width=0.8\textwidth]{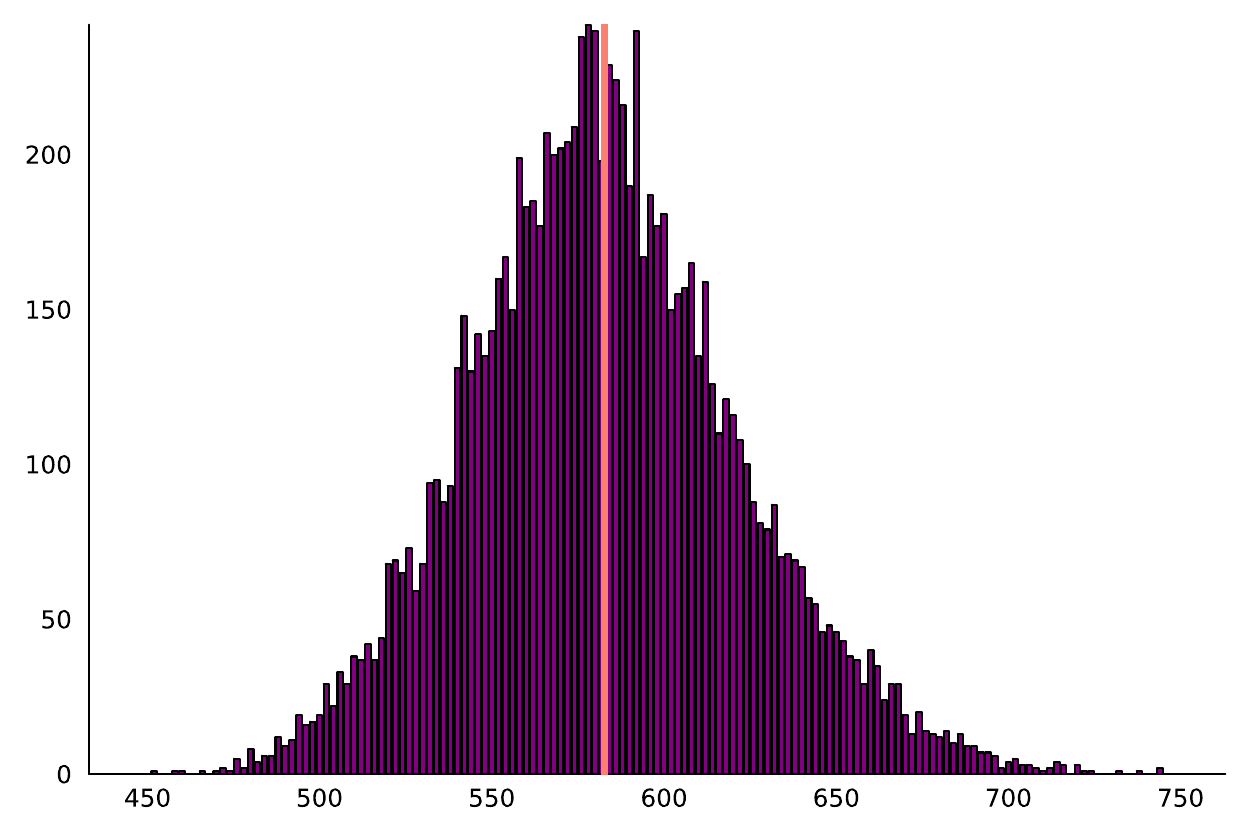}
    \label{fig:panel_b}
  \end{subfigure}
  \caption{
    Empirically observed real critical point counts of the log-likelihood function on $\mathrm{Gr}(2,n)$ for data sampled uniformly at random from the probability simplex over $10\,000$ independent trials. Bar heights show the frequency of each count.
    The average number of real critical points is shown in pink.
  }\label{fig:realcp}
  
\end{figure}

The critical points of the log-likelihood function on the Grassmannian are not all real.
We compute the number of real critical points for $d = 2$ and $n = 4,5,6,7$ for data sampled uniformly at random from the simplex $\Delta_{\binom{n}{2}-1}$. 
These computations were performed with \verb+HomotopyContinuation.jl+ \cite{hc}.
The real critical point counts for $n = 6,7$ are shown with their frequencies in Figure~\ref{fig:realcp}.
The average number of real and positive critical points is shown in Table~\ref{tab:avg} for $n = 4,5,6,7$.
We also show the percentage of our sample points that had more than $1$ critical point in the positive Grassmannian.
We did not observe any data points with more than three positive critical points. 
\begin{table}[b]
  \[
  \begin{array}{ccccc}
    \toprule
    n & 4 & 5 & 6 & 7 \\
    \midrule
    \textrm{ML Degree} & 4 & 22 & 156 & 1\,368 \\
    \textrm{Average No. Real Critical Points} & 2.1004 & 12.3824 & 80.3898& 582.6868\\
    \textrm{Average No. Positive Critical Points} & 1.0096 & 1.0076 & 1.0016 & 1.0016\\
    \textrm{Data with 3 Positive Critical Points} & 0.48\% & 0.38\% & 0.08\% & 0.08\%\\
    \bottomrule
  \end{array}
  \]
  \caption{ML degrees and average numbers of real and positive critical points for $d = 2$ over $10\,000$ trials sampling data uniformly at random from the probability simplex.
    The last row shows the percentage of sample points that had more than one positive critical point. 
  }\label{tab:avg}
\end{table}

%% In these experiments, we observe that the number of real critical points increases when the vector $u$ degenerates.
%% For $n = 4$ and $n = 5$, data points $(u_{ij})_{1 \leq i < j \leq n} \in \Delta_{\binom n 2 -1}$ in the simplex that are close to the vertices of the simplex indexed by pairs $\{i, (i\bmod n)+1\}$ have $4$ or $22$ critical points, respectively.
%% We observe this by solving the optimization problem for the data vectors with $u_{i,(i\bmod n)+1} = 1000$ and $u_{kj} = 1$ for all other pairs. 
%% For $n = 6$, we observe only $130$  real critical points (out of $156$ complex points) for the same data vector.

In the first version of this article, we conjectured that for $d = 2$, given positive data, the log-likelihood function has at most one critical point on the positive Grassmannian $\mathrm{Gr}(2,n)_{>0}$.
Counterexamples for $n = 7,8$ were quickly found by Alexander Urim and Andrew Leber using large language models.
We thank them for their communications.

In what follows, we describe a family of counterexamples in $\mathrm{Gr}(2,4)$. 
The Grassmannian $\mathrm{Gr}(2,4)$ is a hypersurface in $\mathbb P^5$ cut out by the Pl\"ucker quadric
\begin{align*}
  x_{12}x_{34} - x_{13}x_{24} + x_{14}x_{23}.
\end{align*}
For generic, complex data $u$, the maximum likelihood estimation problem
\begin{equation}\label{eq:opt}
  \max_{x \in \mathrm{Gr}(2,4)} \sum_{1 \leq i < j \leq 4} u_{ij}\log(x_{ij}) -
  (\sum_{1 \leq i < j \leq 4}	u_{ij}) \log  (\sum_{1 \leq i < j \leq 4}   x_{ij})
\end{equation}
has $4$ complex critical points by Theorem~\ref{thm:grass-ml-deg}. 
The number of real critical points of a data point $u \in \Delta_5^\circ$ depends on where $u$ lies relative to the \emph{logarithmic discriminant} 
\begin{align*}
  \overline{
    \big \{u \in \mathbb P^{5} : \textrm{the critical point set of } \eqref{eq:opt} \textrm{ is infinite or nonreduced}\big \}}.
\end{align*}
The logarithmic discriminant was first introduced for linear models in \cite{KKT25}, and it was studied for the moduli space $\mathcal M_{0,n}$ in \cite{BBB,KKT25}.
The logarithmic discriminant of $\mathrm{Gr}(2,4)$ is an irreducible hypersurface of degree $14$ in the data space $\PP^5$. 
Its defining polynomial has $8\,427$ terms:
\begin{gather}
  \scriptstyle
  \nonumber
    u_{12}^8u_{13}^2u_{14}^2u_{23}^2 + 2u_{12}^8u_{13}u_{14}^3u_{23}^2 + u_{12}^8u_{14}^4u_{23}^2 + 8u_{12}^7u_{13}^3u_{14}^2u_{23}^2 + 16u_{12}^7u_{13}^2u_{14}^3u_{23}^2\\  \scriptstyle
    \label{eq:disc} 
    + 8u_{12}^7u_{13}u_{14}^4u_{23}^2 + 26u_{12}^6u_{13}^4u_{14}^2u_{23}^2 + 48u_{12}^6u_{13}^3u_{14}^3u_{23}^2 + 16u_{12}^6u_{13}^2u_{14}^4u_{23}^2 + 44u_{12}^5u_{13}^5u_{14}^2u_{23}^2\\ \nonumber  \scriptstyle
    + 64u_{12}^5u_{13}^4u_{14}^3u_{23}^2 - 16u_{12}^5u_{13}^3u_{14}^4u_{23}^2 + 41u_{12}^4u_{13}^6u_{14}^2u_{23}^2 + 26u_{12}^4u_{13}^5u_{14}^3u_{23}^2 - 95u_{12}^4u_{13}^4u_{14}^4u_{23}^2 + \cdots
  \end{gather}
The full discriminant was computed in \verb+Oscar.jl+ \cite{oscar} and is available on \verb+Zenodo+.
For data in a connected component of the complement of the logarithmic discriminant, the number of real critical points of \eqref{eq:opt} is constant.
The next result shows that the logarithmic discriminant also governs the number of positive solutions.
\begin{proposition}
  If two data points $u, u' \in \Delta_5^\circ$ are in the same connected component of the complement of \eqref{eq:disc} intersected with $\Delta_5^\circ$, then \eqref{eq:opt} has the same number of positive critical points for $u$ and $u'$. 
\end{proposition}
\begin{proof}
  By assumption, there is a continuous path $\gamma: [0,1] \to \Delta_5^\circ$ with $\gamma(0) = u$ and $\gamma(1) = u'$ such that \eqref{eq:disc} does not vanish on $\gamma(t)$ for any $t \in [0,1]$.
  The path $\gamma$ induces a path $\Gamma : [0,1] \to \mathrm{Gr}(2,4)^4$ such that $\Gamma_1(t), \ldots, \Gamma_4(t)$ are the $4$ critical points of \eqref{eq:opt} for the data $\gamma(t)$. 
  Since $\gamma$ does not meet the logarithmic discriminant, the path $\Gamma$ preserves the reality of each critical point.
  Thus the number of positive critical points can differ only if $\Gamma_i(t)$ has a zero coordinate for some $i \in \{1,2,3,4\}$ and $t \in [0,1]$. 
  In other words, the number of positive critical points can change only if there exists $t \in [0,1]$ such that $\gamma(t)$ is in the set
  \begin{align}\label{eq:atypical}
  \overline{
    \big \{u \in \mathbb P^{5} : \textrm{a critical point of \eqref{eq:opt} has a zero coordinate} \big \}}.
  \end{align}
  Since $\gamma$ is contained in the open simplex $\Delta_5^\circ$, it suffices to prove that \eqref{eq:atypical} does not intersect $\Delta_5^\circ$ in codimension $1$. 
  By symbolically intersecting \eqref{eq:atypical} with a general line, we observe that it is the union of $35$ hyperplanes.
  A computation in \verb+Oscar.jl+ \cite{oscar} shows that none of these hyperplanes meet the relative interiors of any edges of the simplex.
  Thus they cannot intersect $\Delta_5^\circ$ and we conclude that the number of positive critical points is constant on $\gamma$. 
\end{proof}
\begin{table}[b]
  \begin{center}
  \begin{footnotesize}
  \begin{tabular}{cccccccc}
    \toprule
    & Region 1 & Region 2 & Region 3 & Region 4 & Region 5 & Region 6\\
    \midrule
    Boundary points & -- & $e_{12}$ & $ e_{23}$ & $ e_{34}$ & $e_{14}$ & $(e_{13} + e_{24})/2$ \\
    No. Real Critical points & 2 & 4 & 4 & 4 & 4 & 4\\
    No. Positive Critical points & 1 & 1 & 1 & 1 & 1 & 3\\
    \bottomrule
  \end{tabular}
    \end{footnotesize}
  \end{center}
  \caption{Data regions in the probability simplex $\Delta_5$.
    Region 1 fills most of the simplex.
    Each remaining region is a small open set in $\Delta_5^\circ$ whose Euclidean closure contains the boundary point in the first row of the table. 
  }\label{tab:regions}
\end{table}

By sampling data points, we have identified 6 regions in the complement of the logarithmic discriminant; see Table~\ref{tab:regions}.
For data in these regions, we observe three different phenomena: \eqref{eq:opt} has two real critical points, one of which is positive, \eqref{eq:opt} has four real critical points, one of which is positive, or \eqref{eq:opt} has four real critical points, three of which are positive.

We now turn our attention to Region 6 in Table~\ref{tab:regions}.
For data in this region, \eqref{eq:opt} has three positive critical points. 
To visualize this region, we slice our data space by setting $r = u_{12} = u_{34}$, $s = u_{13} = u_{24}$, and $t = u_{14} = u_{23}$.
The Pl\"ucker quadric in these new variables is $r^2 + t^2 - s^2$.
Restricted to this linear slice, the logarithmic discriminant \eqref{eq:disc} is
\begin{gather*}
  \scriptstyle
  r^6 + 8r^5s + 10r^5t + 36r^4s^2 + 56r^4st + 31r^4t^2 + 68r^3s^3 + 136r^3s^2t + 128r^3st^2 + 44r^3t^3 \\
  \scriptstyle  
  + 53r^2s^4 - 100r^2s^3t - 56r^2s^2t^2 + 128r^2st^3 + 31r^2t^4 + 12rs^5 - 178rs^4t - 100rs^3t^2 \\
  \scriptstyle  
  + 136rs^2t^3 + 56rst^4 + 10rt^5 - 2s^6 + 12s^5t + 53s^4t^2 + 68s^3t^3 + 36s^2t^4 + 8st^5 + t^6.
\end{gather*}
There is an additional linear factor which does not intersect $\Delta_2^\circ$. 
The logarithmic discriminant divides $\Delta_2$ into two regions.
One region contains the Grassmannian, and for data in this region \eqref{eq:opt} has a unique positive critical point.
A data point in the other region of the discriminant complement has three critical points on the Grassmannian; see Figure~\ref{fig:3crit}.
Since all three positive critical points are reduced, the positivity persists on an open region of $\Delta_5^\circ$. 
By sampling data from $\Delta_2$ and evaluating the discriminant, we find that the area of the portion above the discriminant is about 2.5\% of the total area of the simplex.

\begin{figure}[h! tbp]
  \begin{center}
    \includegraphics[width=0.4\textwidth]{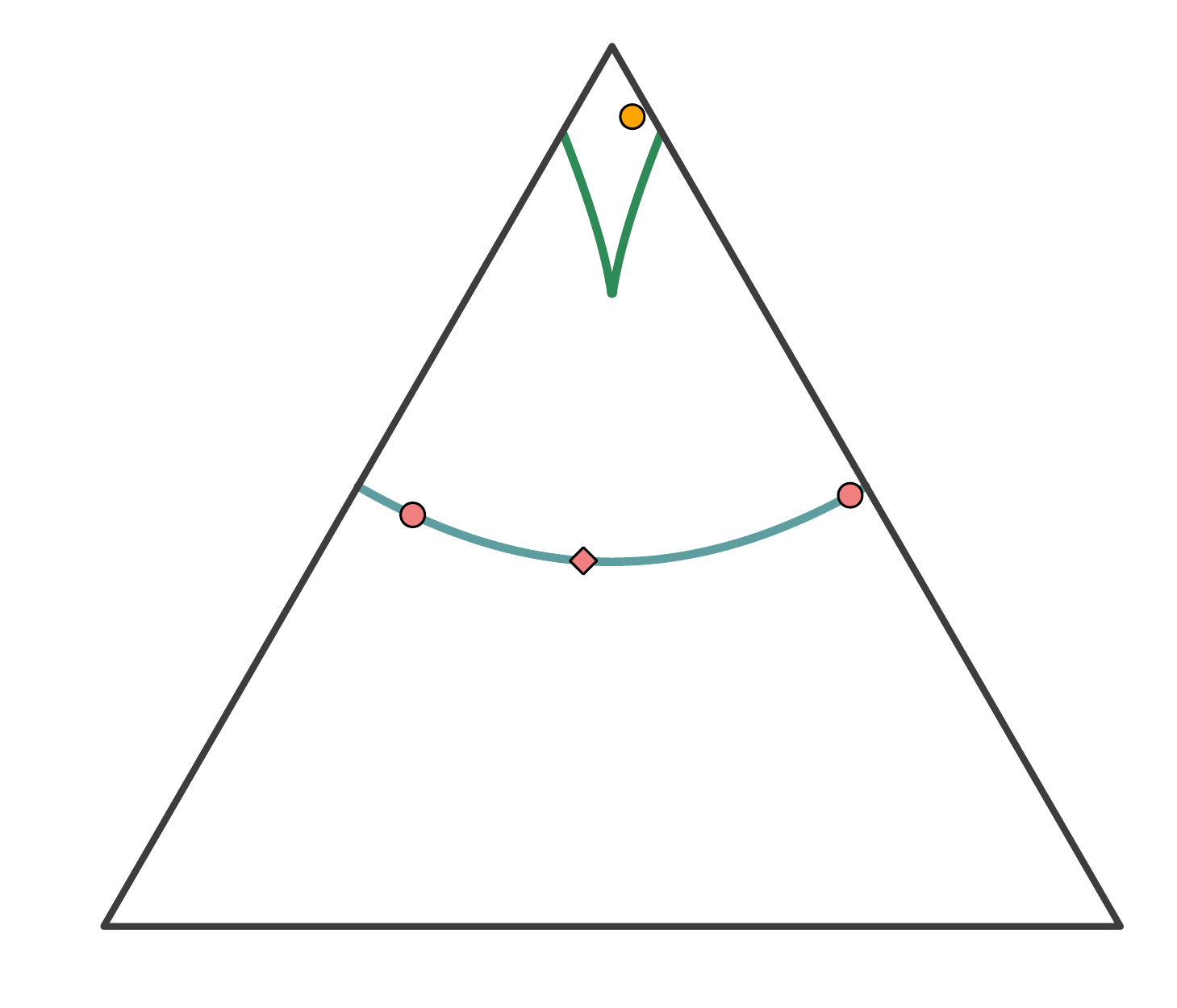}
  \end{center}
  \caption{
    The logarithmic discriminant (green) and Grassmannian $\mathrm{Gr}(2,4)$ (blue) in the slice $r = u_{12} = u_{34}$, $s = u_{13 } = u_{24}$ and $t = u_{14} = u_{23}$ of $\Delta_5$.
    The orange data point $(r,s,t) = (0.02,0.92,0.06)$ above the discriminant has three reduced positive critical points shown in pink on the Grassmannian.
    The two circles are local maxima and the diamond is a saddle point.
    The global maximum $(0.02, 0.49, 0.49)$ is the point furthest to the right.}\label{fig:3crit}
\end{figure}

Our calculations show that for the majority of data points in the simplex $\Delta_5$, the maximum likelihood estimation problem \eqref{eq:opt} has one positive critical point, and that there is a small, full dimensional region in $\Delta_5$ where \eqref{eq:opt} has more than one positive critical point.

\bigskip

{\centering
   {\bf Acknowledgments}\\}
\noindent
We thank Simon Telen for helpful conversations about the proof of Theorem~\ref{thm:grass-ml-deg}, and Serkan Ho\c{s}ten and Bernd Sturmfels for feedback on previous drafts.
We are~grateful~for~the supportive research environment at the Max Planck Institute for Mathematics in the Sciences.

\begin{small}
  
\end{small}
\noindent Hannah Friedman, UC Berkeley
\hfill \url{hannahfriedman@berkeley.edu}


\begin{thebibliography}{10}
    \setlength{\itemsep}{-0.6mm}
\bibitem{ABFKST} D.~Agostini, T.~Brysiewicz, C.~Fevola, L.~K\"uhne,  B.~Sturmfels and S.~Telen:
  {\em Likelihood degenerations},
  Advances in Mathematics~{\bf 414} (2023) 108863.
\bibitem{ABB} C.~Amendola, N.~Bliss, I.~Burke, C.R.~Gibbons, M.~Helmer, S.~Ho\c{s}ten, E.D.~Nash, J.I.~Rodriguez and D.~Smolkin,
  {\em The maximum likelihood degree of toric varieties},
  Journal of Symbolic Computation, {\bf 92} (2019) 222--242.   
\bibitem{BBB} B.~Betti, V.~Borovik, B.~Finkel, B.~Sturmfels and B.~Zacovic:
  {\em Graphical scattering equations},
  {\tt arXiv:2511.07316}. 
\bibitem{hc} P.~Breiding and S.~Timme:
  {\em HomotopyContinuation.jl: A package for homotopy continuation in Julia}
  in {\em Mathematical Software -- ICMS 2018},  Lecture Notes in Computer Science, Vol. 10931, 458--465, Springer, Cham (2018).
\bibitem{CHKS}
  F.~Catanese, S.~Ho\c{s}ten, A.~Khetan and B.~Sturmfels:
  {\em The maximum likelihood degree},
  American Journal of Mathematics {\bf 128} (2006) 671--697.
\bibitem{CHKO}
  O.~Clarke, S.~Ho\c{s}ten, N.~Kushnerchuk and J.~Oldekop:
  {\em Matroid stratifications of ML degrees of independence models},
  Algebraic Statistics {\bf 15} (2024) 199--223.
\bibitem{ACM}
  M.A.~de~Cataldo and L.~Migliorini:
  {\em The {H}odge theory of algebraic maps},
  Annales Scientifiques de l'\'Ecole Normale Sup\'erieure.
  Quatri\`eme S\'erie {\bf 38} (2005) 693--750.
\bibitem{DFRS}
  K.~Devriendt, H.~Friedman, B.~Reinke and B.~Sturmfels:
  {\em The two lives of the Grassmannian},
  Acta Universitatis Sapientiae, Mathematica {\bf 17} (2025) 8.
\bibitem{friedman2025}
  H.~Friedman:
  {\em Likelihood geometry of the squared Grassmannian},
  Proceedings of the American Mathematical Society {\bf 153} (2025) 4463--4474.
\bibitem{GKZ}
  I.M.~Gel$'$fand, M.M.~Kapranov and A.V.~Zelevinsky:
  \emph{Discriminants, resultants, and multidimensional determinants},
  Mathematics: Theory \& Applications, Birkh\"auser Boston, Inc., Boston, MA, 1994.
\bibitem{HS18}
  M.~Helmer and B.~Sturmfels:
  {\em Nearest points on toric varieties},
  Mathematica Scandinavica {\bf 122} (2018) 213--238.
\bibitem{HKS}
  S.~Ho\c{s}ten,  A.~Khetan and B.~Sturmfels:
  {\em Solving the likelihood equations},
  Foundations of Computational Mathematics {\bf 5} (2005) 389--407.
\bibitem{HS}
  J.~Huh and B.~Sturmfels:
  {\em Likelihood geometry} in
  {\em Combinatorial Algebraic Geometry},
  Lecture Notes in Mathematics, Vol. 2108, 63--117, Springer, Cham, (2014).
\bibitem{KKT25}
  L.~Kayser, A.~Kretschmer and S.~Telen:
  {\em Logarithmic discriminants of hyperplane arrangements},
  Le Matematiche {\bf 80} (2025) 325--346.
\bibitem{oscar}
  OSCAR -- Open Source Computer Algebra Research system, Version 1.7.2,
  The OSCAR Team, 2026. (\verb+https://www.oscar-system.org+)  
\bibitem{siva2011}
  S.~Sivasubramanian:
  {\em Signed excedance enumeration via determinants},
  Advances in Applied Mathematics {\bf 47} (2011) 783--794.
\bibitem{ST}
  B.~Sturmfels and S.~Telen: {\em Likelihood equations and scattering amplitudes},
  Algebraic~Statistics~{\bf 12} (2021) 167--186.
\bibitem{sullivant}
  S.~Sullivant:
  {\em Algebraic Statistics},
  Graduate Studies in Mathematics, Vol. 194, American
Mathematical Society, Providence, RI, 2018.
\bibitem{williams2023}
  L.~Williams:
  {\em The positive Grassmannian, the amplituhedron, and cluster algebras},
in {\em ICM--International Congress of Mathematicians}, 
Vol. 6, 4710--4737, EMS Press, Berlin (2023).
\end{thebibliography}
\end{document}